\documentclass[11pt,reqno]{amsart}
\newcommand{\mysection}[1]{\section{#1}
      \setcounter{equation}{0}}

\newcommand\cbrk{\text{$]$\kern-.15em$]$}} 
\newcommand\opar{\text{\raise.2ex\hbox{${\scriptstyle | }$}\kern-.34em$($} }

\usepackage{color}
\usepackage{amsmath,amsthm,amssymb,amsfonts,enumerate,color,enumerate}
\usepackage[pdftex]{graphicx}

\usepackage{verbatim}
\usepackage[spanish,USenglish]{babel} 
\usepackage{color}
\usepackage{tikz}
\allowdisplaybreaks

\DeclareMathOperator*{\esssup}{\pi-ess\,sup}

\newtheorem{theorem}{Theorem}[section]
\newtheorem{lemma}[theorem]{Lemma}

\newtheorem{corollary}[theorem]{Corollary}

\theoremstyle{definition}
\newtheorem{assumption}{Assumption}[section]

\newtheorem{example}{Example}[section]

\theoremstyle{remark}
\newtheorem{remark}{Remark}[section]

\newcommand\bL{\mathbb{L}}
\newcommand\bR{\mathbb{R}}

\newcommand\cB{\mathcal{B}}

\newcommand\cF{\mathcal{F}}

\newcommand\cK{\mathcal{K}}
\newcommand\cL{\mathcal{L}}

\newcommand\cP{\mathcal{P}}

\newcommand\cU{\mathcal{U}}

\newcommand\cZ{\mathcal{Z}}

\makeatletter
 \newcommand{\sumstar}
 {\operatornamewithlimits{\sum@\kern-.2em\raise1ex\hbox{*}}}
 \makeatother

\begin{document}

\author[I. Gy\"ongy]{Istv\'an Gy\"ongy}
\address{School of Mathematics and Maxwell Institute, 
University of Edinburgh, Scotland, United Kingdom.}
\email{i.gyongy@ed.ac.uk}

\author[S. Wu]{Sizhou Wu}
\address{School of Mathematics,
University of Edinburgh,
King's  Buildings,
Edinburgh, EH9 3JZ, United Kingdom}
\email{Sizhou.Wu@ed.ac.uk}



\begin{abstract} A well-known It\^o formula for 
finite dimensional processes, given in terms of stochastic integrals with respect to 
Wiener processes and Poisson random measures, is revisited and 
is revised. The revised formula, which corresponds to the classical It\^o formula for semimartingales
with jumps, is then used to obtain a generalisation of an important infinite dimensional It\^o formula  
for continuous semimartingales from  
Krylov \cite{K2010} to a class of $L_p$-valued 
jump processes. This generalisation is motivated by applications in the theory of stochastic 
PDEs. 
\end{abstract}

\title[It\^o formulas]{On It\^o formulas for jump processes}

\maketitle

\mysection{Introduction}
This is a review paper on some It\^o formulas in finite and infinite dimensional 
spaces. First we consider finite dimensional It\^o-L\'evy processes, which are 
$\bR^M$-valued stochastic processes $X=(X_t)_{t\geq0}$ given in terms of 
stochastic integrals with respect to Wiener processes and Poisson random measures.  
They play important roles in modelling stochastic phenomena when jumps may occur at random times, 
see for example, \cite{CT} and \cite{DeMa}. Chain rules, 
called It\^o formulas, for their transformations $\phi(X_t)$ by sufficiently smooth functions 
$\phi$ are basic tools in the investigations of the stochastic phenomena modelled by It\^o-L\'evy processes, 
see, e.g.,  \cite{KLL} and the references therein. 
It is therefore important to have It\^o formulas for large classes of processes $X$ and functions 
$\phi$. 
Note that classical
It\^o's formula, \eqref{formula standard} below, 
holds only under some restrictive 
conditions, which are not satisfied in important applications, for example in applications 
to filtering theory of partially observed jump diffusions. Therefore we revisit the 
chain rule \eqref{formula standard} for finite dimensional It\^o-L\'evy processes,  
discuss its limitations, and derive formula \eqref{Ito01} from it, 
which corresponds to a well-known It\^o formula 
for general semimartingales, and is valid without restrictive 
conditions on the It\^o-L\'evy processes $X$  
and on the functions $\phi$. 

In the second part of the paper we discuss 
infinite dimensional generalisations of the It\^o formula \eqref{Ito01} from point 
of view of applications in stochastic PDEs (SPDEs).  
In the theory of 
parabolic SPDEs, arising in nonlinear filtering theory, 
the solutions $v=v_t(x)$ of SPDEs 
have the stochastic differentials 
\begin{equation}                                                                   \label{0}
dv_t(x)=(f^0_t(x)+\sum_{i=1}^d\tfrac{\partial}{\partial x^i} f^i_t(x))\,dt
+\sum_{r}g^{r}_t(x)\,dm_t^r 
\end{equation}
with appropriate random functions $f^{\alpha}$ and $g^r$ of $t\in[0,T]$ 
and $x=(x^1,...,x^d)\in\bR^d$, and 
a sequence of martingales $(m^i)_{i=1}^{\infty}$ is.  This stochastic differential is understood 
in a weak sense, i.e., for each smooth function $\varphi$ with compact support on $\bR^d$ 
we have the stochastic differential 
$$
d(v_t,\varphi)=((f^0_t,\varphi)-\sum_i(f^i_t,\tfrac{\partial}{\partial x^i}\varphi))\,dt
+\sum_{r}(g^{r}_t,\varphi)\,dm_t^r, 
$$
where $(u,v)$ denotes the Lebesgue integral over $\bR^d$ of the product $uv$   
for functions $u$ and $v$ of $x\in\bR^d$. In the $L_2$-theory of SPDEs 
$f^{\alpha}$ and $g^r$ are $L_2(\bR^d,\bR)$-valued functions of $(\omega,t)$, satisfying appropriate 
measurability conditions, and to get a priori estimates, 
a suitable formula for $|v|^2_{L_2}$ plays crucial roles. Such a formula in an abstract setting was 
first obtained in \cite{P1975} when $(m^i)_{i=1}^{\infty}$ is a sequence of independent 
Wiener processes. 
The proof in \cite{P1975} is connected with the theory of SPDEs developed in 
\cite{P1975}.      
A direct proof was given in \cite{KR1981}, which was generalised in \cite{GK1982} 
to the case of square integrable martingales $m=(m^i)$.  A nice short proof was presented 
in \cite{K2013}, and 
further generalisations can be found, for example,  in \cite{GS2017} and \cite{PR}. 
The above results  on It\^o formula 
are  
used in the $L_2$-theory of linear and nonlinear SPDEs to obtain existence, uniqueness and regularity results 
under various assumptions 
see, e.g., \cite{G1982}, \cite{KR1981},  \cite{P1975}, 
\cite{PR} and \cite{Ro}.  
To have a similar tool for studying solvability, uniqueness and regularity problems 
for solutions in $L_p$-spaces for $p\neq2$ one should establish a suitable formula  
for $|v_t|^p_{L_p}$, which was first achieved in Krylov \cite{K2010} for $p\geq2$ 
when $(m^{i})_{i=1}^{\infty}$ is a sequence of independent Wiener processes. 

In section \ref{section L_p Ito} we present a generalisation of the main result from Krylov \cite{K2010} 
to the case when the stochastic differential of $v_t$ is of the form 
\begin{equation}                                                                   \label{00}
dv_t(x)=(f^0_t(x)+\sum_{i=1}^d\tfrac{\partial}{\partial x^i} f^i_t(x))\,dt
+\sum_{r}g^{r}_t(x)\,dw_t^r+\int_{Z}h_t(z,x)\tilde \pi(dz,dt),  
\end{equation}
where $\tilde\pi(dz,dt)$ is a Poisson martingale measure with a $\sigma$-finite characteristic measure 
$\mu$ on a measurable space $(Z,\cZ)$, and $h$ is a function on 
$\Omega\times[0,T]\times Z\times \bR^d$. This is  Theorem \ref{theorem1} below, which is 
a slight generalisation of Theorem 2.2  on It\^o's formula from \cite{GW Ito} for $|v_t|^p_{L_p}$ 
for $p\geq2$. We prove it 
by adapting ideas and methods from Krylov \cite{K2010}. In particular, we 
use the finite dimensional It\^o's formula \eqref{Itop} below 
for $|v^{\varepsilon}_t(x)|^p$ for each $x\in\bR^d$, where $v_t^{\varepsilon}$ is an  
approximation of $v_t$ obtained by smoothing it in $x$. Hence we integrate both sides of the formula 
for $|v^{\varepsilon}_t(x)|^p$ over $\bR^d$, change the order of deterministic and stochastic integrals, 
integrate by parts in terms containing derivatives of smooth approximations of $f^i$, and finally 
we let $\varepsilon\to0$. Though the idea of the proof is simple, there are several technical difficulties 
to implement it. We sketch the proof of Theorem \ref{theorem1} in section \ref{section L_p Ito}, 
further details of the proof can be found in 
\cite{GW Ito}. Theorem \ref{theorem1} plays a crucial role 
in proving existence, uniqueness and regularity results in \cite{GW2019} 
for solutions to stochastic integro-differential equations. In \cite{GW2019} 
instead of a single random field $v_t(x)$ we have to deal with a system of random fields 
$v^i_t(x)$ for $i=1,2,...,M$, and we need estimates for $||\sum_{i}|v^i|^2|^{1/2}|_{L_p}$. 
This is why in Theorem \ref{theorem1} we consider a system a random fields $v^i$, $i=1,2,...,M$.

There are known theorems in the literature on It\^o's 
formula for semimartingales with values 
in separable Banach spaces, see for example,  \cite{BNV}, \cite{R2006}, \cite{VV}, 
\cite{ZBH} and \cite{ZBL}.  
In some directions these results are more general than Theorem \ref{theorem1}, but 
they do not cover it.  
In \cite{BNV} and \cite{VV}  
only continuous semimartingales are considered and their differential 
does not contain $D_if^i\,dt$ 
terms.  In \cite{R2006},  \cite{ZBH} and \cite{ZBL} semimartingales 
containing stochastic integrals  
with respect to Poisson random measures and 
martingale measures are considered, but they do not contain terms corresponding to 
$D_if^i$.  Thus the It\^o formula in these papers cannot be applied 
to $|v_t|_{L_p}^p$ when the stochastic differential $d v_t$ is given by 
\eqref{00}. 

In conclusion we present some notions and notations. 
All random elements are given on a fixed complete probability space 
$(\Omega,\cF,P)$ equipped with a right-continuous filtration $(\cF_t)_{t\geq0}$ 
such that $\cF_0$ contains all $P$-zero sets of $\cF$. The $\sigma$-algebra  
of the predictable subsets of $\Omega\times[0,\infty)$ is denoted by $\cP$. 
We are given a sequence 
$w=(w^1_t,w^2_t,...)_{t\geq0}$ of $\cF_t$-adapted independent Wiener processes 
$w^r=(w^r_t)_{t\geq0}$,  such that $w_t-w_s$ is independent of $\cF_s$ for any $0\leq s\leq t$. 
For an integer $m\geq1$ we are given also a sequence of 
independent Poisson random measures $\pi^k(dz,dt)$ 
on $[0,\infty)\times Z^k$, 
with intensity measure $\mu^k(dz)\,dt$ for $k=1,2,..m$, where $\mu^k$ is a $\sigma$-finite measure 
on a measurable space $(Z^k,\cZ^k)$ with a countably generated  $\sigma$-algebra 
$\cZ^k$. We assume that the process $\pi^k_t(\Gamma):=\pi^k(\Gamma\times(0,t])$,  
$t\geq0$, is $\cF_t$-adapted and $\pi^k_t(\Gamma)-\pi^k_s(\Gamma)$ is independent 
of $\cF_s$ for any $0\leq s\leq t$  and $\Gamma\in\cZ^k$ such that $\mu^k(\Gamma)<\infty$. 
We use the notation $\tilde\pi^k(dz,dt)=\pi^k(dz,dt)-\mu^k(dz)dt$ for the {\it compensated 
Poisson random measure}, and set 
$\tilde\pi^k_t(\Gamma)=\tilde\pi^k(\Gamma\times(0,t])=\pi_t^k(\Gamma)-t\mu^k(\Gamma)$ for $t\geq0$ 
and $\Gamma\in \cZ$ 
such that $\mu^k(\Gamma)<\infty$. If $m=1$ then we write 
$\pi$, $\tilde\pi$, $Z$, $\cZ$ and  $\mu$ in place of $\pi^1$, $\tilde\pi^1$, $Z^1$, $\cZ^1$ and $\mu^1$, 
respectively. For basic results concerning stochastic integrals 
with respect to Poisson random measures and Poisson martingale measures  
we refer to \cite{A2009} and \cite{IW2011}. 

Let $M>0$ be an integer. The space of sequences
$\nu=(\nu^{1},\nu^{2},...)$
of vectors $\nu^{k}\in\bR^{M}$ with finite norm
$$
|\nu|_{\ell_{2} }=\big(
\sum_{k=1}^{\infty}|\nu^k|^{2}\big)^{1/2}
$$
is denoted by $\ell_2=\ell_2(\bR^M)$, and by $l_2$ when $M=1$.
We use the notation $D_i$ to denote the $i$-th derivative, i.e.
$$
D_i=\frac{\partial}{\partial x_i},\quad i=1,2,...,M.
$$ 
For vectors $v$ from Euclidean spaces, $|v|$ means the Euclidean norm of $v$. 
The space of smooth functions with compact support in $\bR^M$ is denoted by 
$C^{\infty}_0(\bR^M)$. For integers $k\geq1$ the notation $C^k(\bR^M)$ means 
the space of functions on $\bR^M$ whose derivatives up to order $k$ exist 
and are continuous, and $C_b^k(\bR^M)$ denotes the space of functions on 
$\bR^M$ whose derivatives up to order $k$ are bounded continuous functions. 
When we talk about the derivatives up to order $k$ of a function $f$ then among these 
derivatives we always consider the `` zero order derivative" of $f$, i.e., $f$ itself.

\mysection{It\^o formulas in finite dimensions}                         \label{section formulation}

We consider an $\bR^M$-valued semimartigale 
$X=(X^1_t,...,X^M_t)_{t\geq0}$ given by 
$$
X_t=X_0+\int_0^tf_s\,ds
+\int_0^tg_s^{r}\,dw_s^r
$$
\begin{equation}                                                                          \label{1.9.2}
+\sum_{k=1}^m\int_0^t\int_{Z^k}\bar h^k_s(z)\,\pi^k(dz,ds)
+\sum_{k=1}^m\int_0^t\int_{Z^k}h^k_s(z)\,\tilde{\pi}^k(dz,ds), \quad\text{for $t\geq0$},  
\end{equation}
where $X_0$ is 
an $\bR^M$-valued $\mathcal{F}_0$-measurable random variable, 
$f=(f^i_t)_{t\geq0}$ and $g=(g^{ir}_t)_{t\geq0}$ are predictable 
processes with values in $\bR^M$ and $\ell_2=\ell_2(\bR^M)$, respectively, 
$\bar h^k=(\bar h_t^{ik}(z))_{t\in[0,T]}$ and $h^k=(h_t^{ik}(z))_{t\geq0}$ are 
$\bR^M$-valued $\cP\otimes\cZ$-measurable functions  
on $\Omega\times\bR_+\times Z$ for every $k=1,2,...,m$, 
such that almost surely for every $k=1,2,...,m$
\begin{equation}                                                             \label{integrands}
\bar h_t^{ik}(z)h_t^{jk}(z)=0 \quad\text{for $i,j=1,2,...,M$,  
for all $t\geq0$ and $z\in Z$},   
\end{equation}
and 
\begin{equation}                                                             \label{integrals}
\sum_{k=1}^m\left(\int_0^T\int_{Z_k}|\bar h^k_t(z)|\,\pi^k(dz,dt)+
\int_0^T\int_{Z_k}|h^k_t(z)|^2\mu^k(dz)\,dt\right)<\infty,\quad
\int_0^T|f_t|+|g_t|^2_{\ell_2}\,dt<\infty
\end{equation}
for every $T>0$.  Here and later on, unless otherwise indicated, 
the summation convention with respect 
to repeated integer-valued indices is used, 
i.e., $g^r_s\,dw^r_s$ means $\sum_rg^r_s\,dw^r_s$.

The following It\^o's formula is well-known for $m=1$. 
\begin{theorem}                                                            \label{theorem standard}
Let conditions \eqref{integrands} and 
\eqref{integrals} hold and assume there is a constant $K$ such that $|h^k|\leq K$ for all 
$(\omega,t,z)\in\Omega\times \bR_{+}\times Z$ and $k=1,2,...,m$. Then for any 
$\phi\in C^2(\bR^M)$ the process $(\phi(X_t))_{t\geq0}$ is a semimartingale 
such that 
$$
\phi(X_t)=\phi(X_0)
+\int_0^tf^i_sD_i\phi(X_s)+\tfrac{1}{2}g_s^{ir}g_s^{jr}D_{i}D_{j}\phi(X_s)\,ds
+\int_0^tg^{ir}_sD_i\phi(X_s)\,dw^r_s
$$
$$
+\sum_{k=1}^m\int_0^t\int_{Z_k}\phi(X_{s-}+\bar h^k_s(z))-\phi(X_{s-})\,\pi^k(dz,ds)
$$
$$
+\sum_{k=1}^m\int_0^t\int_{Z_k}\phi(X_{s-}+h^k_s(z))-\phi(X_{s-})\,\tilde \pi^k(dz,ds)
$$
\begin{equation}                                                               \label{formula standard}
+\sum_{k=1}^m\int_0^t\int_{Z_k}
\left(
\phi(X_s+h^k_s(z))-\phi(X_s)-h^{ik}_s(z)D_i\phi(X_s)
\right)
\,\mu^k(dz)\,ds 
\end{equation}
holds almost surely for all $t\geq0$. 
\end{theorem} 
\begin{proof}
This theorem, with a finite dimensional Wiener process $w=(w^1,...,w^{d_1})$ 
in place of an infinite sequence of independent Wiener processes  and for $m=1$ 
is proved, for example, in \cite{IW2011}, see Theorem 5.1 in chapter II.  Following this proof 
 with appropriate changes one can easily prove 
the above theorem as follows. 
Since $\mu^k$ is $\sigma$-finite for $k=1,2,...,m$, for each $k$ we have an 
increasing sequence $(Z^k_n)_{n=1}^{\infty}$ of sets  $Z_n^k\in\cZ^k$ such that 
$Z^k=\cup_{n=1}^{\infty}Z_n^k$ 
and $\mu^k(Z_n^k)<\infty$ for every $n$. 
For a fixed integer $n\geq1$ let $\rho^k_1<\rho^k_2<...$ denote the increasing sequence of times 
where the jumps of $N^k:=(\pi^k_t(Z_n^k))_{t\geq 0}$ occur. Similarly, 
let $\tau_1<\tau_2<...$ be the jump times 
of the process $N=\sum_{k=1}^mN_k$. Then $\rho^k_i$ and $\tau_i$ are stopping times 
for every $k=1,2,...,m$ and $i\geq1$, and for almost every $\omega\in\Omega$ the set 
of time points $\{\tau_i(\omega):i\geq1\}$ contains all points of discontinuities of 
$(X^n_t(\omega))_{t\geq0}$, where the process $X^n$ is defined by 
$$
X^n_t=X_0+\int_0^tf_s\,ds
+\int_0^tg_s^{r}\,dw_s^r+\sum_{k=1}^m\int_0^t\int_{Z^k}\bar h^k_s(z){\bf1}_{Z^k_n}(z)\,\pi^k(dz,ds)
$$
\begin{equation}                                                                                         \label{eq_n}
+\sum_{k=1}^m\int_0^t\int_{Z^k}h^k_s(z){\bf1}_{Z^k_n}(z)\,{\pi}^k(dz,ds)
-V^n_t
\quad\text{for $t\geq0$}   
\end{equation}
with 
$$
V^n_t:=\sum_{k=1}^m\int_0^t\int_{Z^k}h^k_s(z){\bf1}_{Z^k_n}(z)\,\mu^k(dz)\,ds. 
$$
Clearly, $\phi(X^n_t)=\phi(X^n_0)+A^n_t+B_t^n$ with 
$$
A^n_t=\sum_{i\geq1}\big(\phi(X^n_{\tau_i\wedge t})-\phi(X^n_{\tau_i\wedge t-})\big),\quad 
B^n_t=\sum_{i\geq1}\big(\phi(X^n_{\tau_i\wedge t-})-\phi(X^n_{\tau_{i-1}\wedge t})\big), 
$$
where we set $\tau_0:=0$ and $X^n_{\tau_i\wedge t-}:=X^n_{\tau_i-}$ for $t\geq\tau_i$ 
and $X^n_{\tau_i\wedge t-}:=X^n_t$ for $t<\tau_i$. 
By It\^o's formula for It\^o processes we have 
$$
\phi(X^n_{\tau_i\wedge t-})-\phi(X^n_{\tau_{i-1}\wedge t})
=\int_{\tau_{i-1}\wedge t}^{\tau_i\wedge t-}D_l\phi(X^n_s)f^l_s
+\tfrac{1}{2}D_{jl}\phi(X^n_s)g^{jr}g^{lr}\,ds
$$
$$
+\int_{\tau_{i-1}\wedge t}^{\tau_i\wedge t-}D_l\phi(X^n_s)g^{lr}_s\,dw^r_s
-\int_{\tau_{i-1}\wedge t}^{\tau_i\wedge t-}D_l\phi(X^n_s)\,dV^n_s, 
$$
which gives 
$$
B^n_t=\int_{0}^tD_l\phi(X^n_s)f^l_s
+\tfrac{1}{2}D_{jl}\phi(X^n_s)g^{jr}g^{lr}\,ds
$$
\begin{equation}                                                                      \label{ContIto}
+\int_{0}^{t}D_l\phi(X^n_s)g^{lr}_s\,dw^r_s
-\int_{0}^{t}D_l\phi(X^n_s)\,dV^n_s.  
\end{equation}  
Notice that $\rho^k_i$ has a density with respect to the Lebesgue measure for $i\geq1$,  
and $\rho^k_i$ and $\rho^l_j$ are independent for $k\neq l$. Hence $P(\rho^k_i=\rho^l_j)=0$ 
for $k\neq l$ and positive integers $i,j$. Consequently, for almost every $\omega\in\Omega$
we have $\{\tau_i(\omega):i\geq1\}=\cup_{k=1}^{m}\{\rho^k_i(\omega):i\geq1\}$ such that the sets in the union 
are almost surely pairwise  disjoint.  
Hence, taking also into account condition \eqref{integrands}, we get that almost surely
$$
A^n_t=\sum_{k=1}^m\sum_{i\geq1}
\big(\phi(X^n_{\rho^k_i\wedge t})-\phi(X^n_{\rho^k_i\wedge t-})\big)
=\bar A^n_t+\tilde A^n_t
$$
for all $t\geq0$, where 
$$
\bar A^n_t=\sum_{k=1}^m
\int_0^t\int_{Z^k}\big(\phi(X^n_{s-}+\bar h^k_s(z))-\phi(X^n_{s-})\big)
{\bf1}_{Z^n_k}(z)\,\pi^k(dz,ds), 
$$
$$
\tilde A^n_t=\sum_{k=1}^m
\int_0^t\int_{Z^k}\big(\phi(X^n_{s-}+h^k_s(z))-\phi(X^n_{s-})\big)
{\bf1}_{Z^n_k}(z)\,\pi^k(dz,ds) 
$$
$$
=\sum_{k=1}^m
\int_0^t\int_{Z^k}\big(\phi(X^n_{s-}+h^k_s(z))-\phi(X^n_{s-})\big)
{\bf1}_{Z^n_k}(z)\,\tilde\pi^k(dz,ds) 
$$
$$
+\sum_{k=1}^m
\int_0^t\int_{Z^k}\big(\phi(X^n_{s-}+h^k_s(z))-\phi(X^n_{s-})\big)
{\bf1}_{Z^n_k}(z)\,\mu^k(dz)\,ds.  
$$
Combining this with \eqref{ContIto} we get 
$$
\phi(X^n_t)=\phi(X_0)+\int_{0}^tD_l\phi(X^n_s)f^l_s
+\tfrac{1}{2}D_{jl}\phi(X^n_s)g^{jr}g^{lr}\,ds+\int_{0}^{t}D_l\phi(X^n_s)g^{lr}_s\,dw^r_s
$$
$$
+\sum_{k=1}^m
\int_0^t\int_{Z^k}\big(\phi(X^n_{s-}+\bar h^k_s(z))-\phi(X^n_{s-})\big)
{\bf1}_{Z^n_k}(z)\,\pi^k(dz,ds), 
$$
$$
+\sum_{k=1}^m
\int_0^t\int_{Z^k}\big(\phi(X^n_{s-}+h^k_s(z))-\phi(X^n_{s-})\big)
{\bf1}_{Z^n_k}(z)\,\tilde\pi^k(dz,ds) 
$$
$$
+\sum_{k=1}^m
\int_0^t\int_{Z^k}\big(\phi(X^n_{s-}+h^k_s(z))-\phi(X^n_{s-})-D_l\phi(X^n_s)h^{lk}_s(z)\big)
{\bf1}_{Z^n_k}(z)\,\mu^k(dz)\,ds.  
$$
Hence we can finish the proof by letting $n\to\infty$ and using standard facts about convergence of 
Lebesgue integrals and stochastic integrals with respect to Wiener processes and random measures. 
\end{proof}

In some publications only the natural conditions \eqref{integrands} and 
\eqref{integrals} are assumed in the formulation of the above theorem, 
but these conditions are not sufficient for 
\eqref{formula standard} to hold, as the following 
simple example shows. 

\begin{example}                                                                                    \label{example1}
Consider a one-dimensional 
semimartingale $(X_t)_{t\in[0,T]}$ given by \eqref{1.9.2} with $f=0$, $g=0$, 
$\bar h=0$ and $h_t(z)={\bf1}_{t>0}t^{-1/4}$, $t\geq0$, $z\in Z=\bR\setminus\{0\}$,  
when $\pi(dz,dt)$ is 
the measure of jumps 
of a standard Poisson process  
and $\tilde\pi(dz,dt)=\pi(dz,dt)-\mu(dz)dt$ 
is its compensated measure, where 
$\mu=\delta_1$ is the Dirac measure on $Z$ concentrated at $1$. Then obviously  
conditions \eqref{integrands} and 
\eqref{integrals} hold, and for $\phi(x)=x^4$ 
the last integrand in \eqref{formula standard} is 
\begin{equation*}
|X_{s-}+h_s(z)|^4-|X_{s-}|^4-4X_{s-}^3h_s(z)
=\sum_{i=1}^3c_i(s,z)
\end{equation*}
with
$$
c_1(s,z)=6|X_{s-}|^2{\bf1}_{s>0}s^{-1/2}, 
\quad
c_2(s,z)=4X_{s-}{\bf1}_{s>0}s^{-3/4}, 
\quad
c_3(s,z)={\bf1}_{s>0}s^{-1}. 
$$
Clearly, 
$$
\int_0^t\int_Z|c_i(s,z)|\,\mu(dz)\,ds<\infty\quad \text{for $i=1,2$}, 
\quad
\text{and} 
\quad
\int_0^t\int_Zc_3(s,z)\,\mu(dz)\,ds=\infty 
$$
for every $t>0$, 
which shows that the last integral in \eqref{formula standard} is infinite. 
Similarly, one can show that almost surely 
$$
\int_0^t\int_Z(|X_{s-}+h_s(z)|^4-|X_{s-}|^4)^2\,\mu(dz)\,ds=\infty
\quad\text{for every $t>0$}, 
$$
 which means the stochastic integral with respect to 
$\tilde\pi(dz,ds)$ in \eqref{formula standard} does not exist. 
\end{example}

It is easy to see that the last two integrals in 
 \eqref{formula standard} are well-defined as It\^o and Lebesgue integrals, 
respectively, under the additional boundedness assumption on $h$. 
Instead of this extra condition  
on $h$ one can make additional assumptions on $\phi$ to ensure that  
formula \eqref{formula standard} holds. 
It is sufficient to assume that the derivatives of $\phi$ up to second order are bounded. 
Such a condition, however, excludes the applicability of It\^o's formula to power functions 
$\phi(x)=|x|^p$ for $p\geq2$. Notice that for any $\phi\in C^2(\bR^M)$ the conditions 
\begin{equation}                                               \label{condition1}
\sum_{k=1}^m\int_0^T\int_{Z_k}|\phi(X_s+h^k_s(z))-\phi(X_s)|^2\,\mu^k(dz)\,ds<\infty
\end{equation}
and 
\begin{equation}                                                \label{condition2}
\sum_{k=1}^m\int_0^T\int_{Z_k}|\phi(X_s+h^k_s(z))-\phi(X_s)-h^k_s(z)\nabla\phi(X_s)|
\,\mu^k(dz)\,ds<\infty
\quad \rm{(a.s.)}
\end{equation}
ensure the existence of the last two integrals in
 \eqref{formula standard} respectively. 
Thus we can expect that under conditions 
\eqref{integrands}-\eqref{integrals} and 
\eqref{condition1}-\eqref{condition2} formula \eqref{formula standard} 
is valid. 

\begin{theorem}                                                                                   \label{theorem ItoBL}
Let conditions \eqref{integrands}-\eqref{integrals} 
and \eqref{condition1}-\eqref{condition2} hold.  
Assume $\phi\in C^2(\bR^M)$.  Then 
$\phi(X_t)$ is a semi-martingale such that \eqref{formula standard} holds  
almost surely for all $t\geq0$.
\end{theorem}
\begin{proof}
This theorem is a slight generalisation of
Theorem 5.2 in \cite{BL1984}. For the convenience of the reader we deduce this theorem 
from Theorem \ref{theorem standard} here. 
For notational simplicity we assume $m=1$, with additional indices 
the case $m>1$ can be proved in the same way. 

For vectors $a=(a^1,....,a^M)\in\bR^M$ and functions $\phi\in C^2(\bR^M)$
we define the functions $I^a\phi$ and $J^a\phi$ by 
\begin{equation}                                                                            \label{def IJ}
I^a\phi(v)=\phi(v+a)-\phi(v),\quad 
J^a\phi(v)=\phi(v+a)-\phi(v)-a^iD_i\phi(v),\quad v\in\bR^M. 
\end{equation}
Assume first $\phi\in C_b^2(\bR^M)$.  Approximate 
$h$ by $h^{(n)}=(h^{1(n)},...,h^{M(n)})$  and define 
$$
X_t^{(n)}=X_0+\int_0^tf_s\,ds
+\int_0^tg_s^{r}\,dw_s^r
+\int_0^t\int_Z\bar{h}_s(z)\,\pi(dz,ds)
+\int_0^t\int_Z h^{(n)}_s(z)\,\tilde{\pi}(dz,ds)
$$
for integers $n\geq1$, where 
$h_t^{i(n)}=-n\vee h_t^{i}\wedge n$. 
Then \eqref{formula standard} holds with $X^{i(n)}_t$ and $h^{i(n)}_t$ 
in place of $X^i_t$ and $h^i_t$, respectively, for each $i=1,2,...,M$.
Clearly, 
$$
\int_0^T\int_Z|h^{(n)}_s(z)-h_s(z)|^2\,\mu(dz)\,ds\rightarrow 0\quad \rm{(a.s.)}
\quad\text{for each $T>0$}, 
$$
which implies
$$
\int_0^t\int_Z h^{(n)}_s(z)\,\tilde{\pi}(dz,ds)\to
\int_0^t\int_Z h_s(z)\,\tilde{\pi}(dz,ds)
$$
in probability uniformly in $t\in[0,T]$.
Consequently for each $T>0$ we have
$$
\sup_{t\in[0,T]}|X^{(n)}_t-X_t|\to 0
$$
in probability.
It is easy to see
$$
\int_0^tf^i_sD_i\phi(X^{(n)}_s)+\tfrac{1}{2}g_s^{ir}g_s^{jr}D_{i}D_{j}\phi(X^{(n)}_s)\,ds
\to\int_0^tf^i_sD_i\phi(X_s)+\tfrac{1}{2}g_s^{ir}g_s^{jr}D_{i}D_{j}\phi(X_s)\,ds,
$$

$$
\int_0^tg^{ir}_sD_i\phi(X^{(n)}_s)\,dw^r_s \to \int_0^tg^{ir}_sD_i\phi(X_s)\,dw^r_s,
$$

$$\int_0^t\int_Z I^{\bar{h}_{s}(z)}\phi(X^{(n)}_{s-})\,\pi(dz,ds)
\to \int_0^t\int_Z I^{\bar{h}_{s}(z)}\phi(X_{s-})\,\pi(dz,ds)$$
in probability uniformly in $t\in[0,T]$ for $T>0$.
Furthermore, by Taylor's formula we have
$$
|J^{h^{(n)}_s(z)}\phi(X^{(n)}_s)| 
\leq  
\int_0^1 (1-\theta)|h^{i(n)}_s(z)h^{j(n)}_s(z)D_{ij}\phi(X^{(n)}_s+\theta h^{(n)}_s(z))|\,d
\leq C|h_s(z)|^2
$$
$$
|I^{h^{(n)}_s(z)}\phi(X^{(n)}_s)|
\leq \int_0^1|\nabla \phi(X^{(n)}_s+\theta h^{(n)}_s(z)) h^{(n)}_s(z)|\,d\theta
\leq C|h_s(z)|^2 
$$
with a constant $C$ independent of $n$.
Hence by Lebesgue's theorem on dominated convergence for $T>0$ we have 
$$
\int_0^T\int_Z |J^{h^{(n)}_s(z)}\phi(X^{(n)}_s)
-J^{h_s(z)}\phi(X_s)|\, \mu(dz)\,ds\to 0\quad\text{for $t\geq0$}
$$
and 
$$
\int_0^T\int_Z|I^{h_s^{(n)}(z)}\phi(X^{(n)}_s)-I^{h_s(z)}\phi(X_s)|^2\,\mu(dz)\,ds
\to 0
\quad\text{for $T\geq0$}
$$
in probability, which implies
$$
\int_0^t\int_ZI^{h^{(n)}_s(z)}\phi(X^{(n)}_{s-})\,\tilde{\pi}(dz,ds)
\to \int_0^t\int_Z I^{h_s(z)}\phi(X_{s-})\,\tilde{\pi}(dz,ds)
$$
in probability uniformly in $t\in[0,T]$ for each $T>0$. 
Hence, letting $n\to \infty$ in \eqref{formula standard} 
with $h^{(n)}$ and $X^{(n)}$ in place of $h$ and $X$, respectively,  
we prove the theorem for $\phi\in C^2_b(\bR^M)$.\\
For $\phi\in C^2(\bR^M)$ we define  $\phi_n$ for integers $n\geq1$ by 
$\phi_n(x)=\phi(x)\zeta(x/n)$, $x\in\bR^M$, 
where $\zeta$ is a smooth function on $\bR^M$ with values in $[0,1]$ such that $\zeta(x)=1$ for 
$|x|\leq 1$ and $\zeta(x)=0$ for $|x|\geq 2$. 
Then  $\phi_n\in C_b^2(\bR^M)$, and therefore \eqref{formula standard} holds with $\phi_n$ 
in place of $\phi$. 
Thus it remains to take limit as $n\to \infty$ for each term in \eqref{formula standard} 
 with $\phi_n$ in place of $\phi$. 
Clearly as $n\to \infty$, we have
$$
\phi_n(x)\to \phi(x),\quad D_i\phi_n(x)\to D_i\phi(x),
\quad D_{ij}\phi_n(x)\to D_{ij}\phi(x)
$$
uniformly on compact subsets of $\mathbb{R}^M$ for $i,j=1,2,...,M$. Hence it is easy to see
$$
\int_0^tf^i_sD_i\phi_n(X_s)+\tfrac{1}{2}g_s^{ir}g_s^{jr}D_{i}D_{j}\phi_n(X_s)\,ds
\to\int_0^tf^i_sD_i\phi(X_s)+\tfrac{1}{2}g_s^{ir}g_s^{jr}D_{i}D_{j}\phi(X_s)\,ds
$$
and
$$
\int_0^tg^{ir}_sD_i\phi_n(X_s)\,dw^r_s \to \int_0^tg^{ir}_sD_i\phi(X_s)\,dw^r_s
$$
in probability, uniformly in $t\in[0,T]$ as $n\to\infty$.
Using the simple identity 
$$
I^{a}(\varphi\phi)(x)=\phi(x)I^{a}\varphi(x)
+\varphi(x+a)I^{a}\phi(x), \quad a,x\in\bR^M
$$
with $\varphi=\phi_n$ and $a=h_s(z)$, we get 
$$
|I^{h_s(z)}\phi_n(X_s)-I^{h_s(z)} \phi(X_s)|
\leq |\phi(X_s)||I^{h_s(z)}\zeta_n(X_s)|
+|1-\zeta_n(X_s+h_s(z))||I^{h_s(z)}\phi(X_s)| 
$$
$$
\leq  \frac{C}{n}|\phi(X_s)||h_s(z)|
+|1-\zeta_n(X_s+h_s(z))||I^{h_s(z)}\phi(X_s)|
$$
\begin{equation}                                                     \label{Ih}
\leq 
 \frac{C}{n}|\phi(X_s)||h_s(z)|
+|I^{h_s(z)}\phi(X_s)|
\end{equation}
with a constant $C$ independent of $n$, and since 
$\lim_{n\to\infty}|1-\zeta_n(X_s+h_s(z))|=0$, we have 
$$
\limsup_{n\to\infty}|I^{h_s(z)}\phi_n(X_s)-I^{h_s(z)} \phi(X_s)|=0
\quad\text{for every $(\omega,s,z)$}. 
$$ 
Hence by \eqref{Ih}, taking into account conditions \eqref {integrals} and \eqref{condition1} 
on $h$ and  $I^{h_s(z)} \phi(X_s)$, we can apply Lebesgue's theorem on 
dominated convergence  to obtain
$$
\lim_{n\to\infty}\int_0^T\int_Z|I^{h_s(z)}\phi(X_s)-I^{h_s(z)}\phi_n(X_s)|^2\,\mu(dz)\,ds
=0 \quad \rm{(a.s.)},
$$
which implies that for $n\to\infty$ we have 
$$
\int_0^t\int_ZI^{h_s(z)}\phi_n(X_s)\,\tilde{\pi}(dz,ds)
\to \int_0^t\int_ZI^{h_s(z)}\phi(X_s)\,\tilde{\pi}(dz,ds)
$$
in probability uniformly in $t\in[0,T]$ for each $T>0$.
Similarly, we get 
$$
\lim_{n\to\infty}\int_0^T\int_Z|
I^{\bar{h}_s(z)}\phi_n(X_s)-
I^{\bar{h}_s(z)}\phi(X_s)|\,\pi(dz,ds)=0\quad \rm{(a.s.)}
$$
for every $T\geq0$. Using the identity 
$$
J^{a}(\varphi\phi)(x)=\phi(x)J^{a}\varphi(x)
+\varphi(x)J^{a}\phi(x)+I^a\varphi I^a\phi, \quad a,x\in\bR^M
$$
with $\varphi=\phi_n$ and $a=h_s(z)$, we get 
$$
J^{h_s(z)}\phi(X_s)-J^{h_s(z)}\phi_n(X_s)
$$
$$
=(1-\zeta_n(X_s))J^{h_s(z)}\phi(X_s)
+\phi(X_s)J^{h_s(z)}\zeta_n(X_s)+I^{h_s(z)}\phi(X_s)I^{h_s(z)}\zeta_n(X_s).
$$
Hence taking into account $|(1-\zeta_n(X_s))|\leq 1$, 
$$
|J^{h_s(z)}\zeta_n(X_s)|
\leq \int_0^1(1-\theta)|h^i_s(z)h^j_s(z)
D_{ij}\zeta_n(X_s+\theta h_s(z))|\,d\theta
\leq \frac{C}{n}|h_s(z)|^2, 
$$
$$
|I^{h_s(z)}\phi(X_s)I^{h_s(z)}\zeta_n(x)|
\leq\frac{C}{n} |I^{h_s(z)}\phi(X_s)||h_s(z)|
\leq \frac{C}{n}( |I^{h_s(z)}\phi(X_s)|^2+|h_s(z)|^2)  
$$
and $\lim_{n\to\infty}|(1-\zeta_n(X_s))|=0$, we obtain 
$$
|J^{h_s(z)}\phi(X_s)-J^{h_s(z)}\phi_n(X_s)|
$$
\begin{equation}                                                          \label{LJ}
\leq 
|J^{h_s(z)}\phi(X_s)|+ \frac{C}{n}(|\phi(X_s)||h_s(z)|^2+ |I^{h_s(z)}\phi(X_s)|^2+|h_s(z)|^2)
\end{equation}
with a constant $C$ independent of $n$, and 
$$
\lim_{n\to\infty}|J^{h_s(z)}\phi(X_s)-J^{h_s(z)}\phi_n(X_s)|=0\quad \text{for all $(\omega,s,z)$}. 
$$
Thus by  virtue of \eqref{LJ} and conditions 
\eqref{integrands}, \eqref{condition1} and  \eqref{condition2} on $h$, $I^h(X_s)$ and $J^h(X_s)$, 
we can use Lebesgue's theorem on dominated convergence again to get 
$$
\lim_{n\to\infty}\int_0^T\int_Z
|J^{h_s(z)}\phi_n(X_s)-J^{h_s(z)}\phi(X_s)|\,\mu(dz)\,ds\quad \rm{(a.s.)} 
$$
for every $T\geq0$, which completes the proof of Theorem \ref{theorem ItoBL}.
\end{proof}

\begin{remark}
The above theorem is useful if one can check that conditions 
\eqref{condition1}-\eqref{condition2} 
are satisfied. If $D_i\phi$ and $D_{ij}\phi$ are bounded functions 
for every $i,j=1,2,...,M$, then conditions \eqref{condition1}-\eqref{condition2} 
are always satisfied, since for every $t>0$ 
\begin{align}
\int_0^t\int_Z|I^{h_s(z)}\phi(X_s)|^2\,\mu(dz)\,ds
&
=\int_0^t\int_Z\Big|\int_0^1\nabla \phi(X_s+h_s(z)) h_s(z)\,d\theta\Big|^2\,\mu(dz)\,ds\nonumber\\
& \leq C\int_0^t\int_Z|h_s(z)|^2\,\mu(dz)\,ds<\infty \quad (\rm{a.s.}) \nonumber
\end{align}
and
\begin{align*}
\int_0^t\int_Z |J^{h_s(z)}\phi(X_s)| \,\mu(dz)\,ds
&= \int_0^t\int_Z|\int_0^1 (1-\theta)h^i_s(z)h^j_s(z)D_{ij}\phi(X_s
+\theta h_s(z))d\theta|\,\mu(dz)\,ds\\
& \leq C\int_0^t\int_Z|h_s(z)|^2\mu(dz)ds<\infty\quad \rm{(a.s.)} 
\end{align*}
with a constant  $C$. Thus by virtue of the above theorem, under the conditions  
\eqref{integrands} and \eqref{integrals} It\^o formula \eqref{formula standard} 
holds if the first and second order derivatives of $\phi$ are bounded continuous functions. 
As Example \ref{example1} shows Theorem \ref {theorem ItoBL} is 
not applicable to $\phi(x)=|x|^p$ 
for $p\geq2$. 
\end{remark}
Next we formulate an It\^o formula which holds under the natural conditions.  
\eqref{integrands}-\eqref{integrals}

\begin{theorem}                                                             \label{theorem0}
Let conditions \eqref{integrands} and \eqref{integrals} hold, 
and let $\phi$ from $C^2(\bR^M)$. 
Then $\phi(X_t)$ is a semimartingale such that 
$$
\phi(X_t)
= \phi(X_0)+ \int_0^tD_i\phi(X_s)g_s^{ir}\,dw_s^r                        
+\int_0^tD_i\phi(X_s)f^i_s+
\tfrac{1}{2}D_iD_j\phi(X_s)g_s^{ir}g_s^{jr}\,ds
$$                                 
\begin{align}
& +\sum_{k=1}^m\int_0^t\int_{Z_k}\phi(X_{s-}+\bar h^k_s(z))-\phi(X_{s-})\,\pi^k(dz,ds)
+\sum_{k=1}^m\int_0^t\int_{Z_k}D_i\phi(X_{s-})h^{ik}_s(z)\,\tilde\pi^k(dz,ds)                                       \nonumber\\
&
+\sum_{k=1}^m\int_0^t\int_{Z_k}
\phi(X_{s-}+h^k_s(z))-\phi(X_{s-})-D_i\phi(X_{s-})h^{ik}_s(z)\,\pi^k(dz,ds)                \label{Ito01}
\end{align}
almost surely for all $t\geq0$. 
\end{theorem}
\begin{proof}
We prove Theorem \ref{theorem0} by rewriting 
It\^o formula \eqref{formula standard} 
into equation \eqref{Ito01} 
under the additional condition that $h$ is bounded, and  then 
we dispense with this condition by approximating 
$h$ by bounded functions.  For notational simplicity  
we assume $m=1$, for $m>1$ the proof goes in the same way. 
First in addition to the conditions \eqref{integrands} and 
\eqref{integrals} assume there is a constant $K$ such that 
$|h|\leq K$. By Taylor's formula 
for 
$I^{a}\phi(v)$
and
$J^{a}\phi(v)$, introduced in \eqref{def IJ}, 
for each $v,a\in\bR^M$ we have 
\begin{equation}                                                               \label{Taylor}
|I^a\phi(v)|\leq \sup_{|x|\leq |a|+|v|}|D\phi(x)||a|, 
\quad |J^a\phi(v)|\leq \sup_{|x|\leq |a|+|v|}|D^2\phi(x)||a|^2,  
\end{equation}
where $|D\phi|^2:=\sum_{i=1}^M|D_i\phi|^2$ and 
$|D^2\phi|^2:= \sum_{i=1}^M\sum_{j=1}^M|D_iD_j\phi|^2$. 
Since $(X_t)_{t\geq0}$ is a cadlag process, $R:=\sup_{t\leq T}|X_t|$ 
is a finite random variable for each fixed $T$. Thus we have  
\begin{equation}                                                             \label{Jintegral1}
\int_0^T\int_Z|J^{h_t(z)}\phi(X_{t-})|\mu(dz)\,dt
\leq \sup_{|x|\leq R+K}|D^2\phi(x)
|\int_0^T\int_Z|h_t(z)|^2\,\mu(dz)\,dt<\infty 
\end{equation}
and 
\begin{equation}                                                            \label{Jintegral2}
\int_0^T\int_Z|J^{h_t(z)}\phi(X_{t-})|^2\mu(dz)\,dt
\leq \sup_{|x|\leq R+K}|D^2\phi(x)|^2K^2
\int_0^T\int_Z|h_t(z)|^2\,\mu(dz)\,dt<\infty 
\end{equation}
almost surely.  Clearly, 
$$
\int_0^T\int_Z|D_i\phi(X_{t-})h^i_t(z)|^2\,\mu(dz)\,dt
\leq \sup_{|x|\leq R}|D\phi(x)|^2
\int_0^T\int_Z|h_t(z)|^2\,\mu(dz)\,dt<\infty \,\,(\rm{a.s.}).   
$$
Hence, by virtue of \eqref{Jintegral2} the stochastic It\^o integral 
$$
\int_0^t\int_Z\phi(X_{s-}+h_t(z))-\phi(X_s)\,\tilde\pi(dz,ds)
=\int_0^t\int_ZI^{h_s(z)}\phi(X_{s-})\tilde\pi(dz,ds)
$$
can be decomposed as
$$
\int_0^t\int_ZI^{h_s(z)}\phi(X_{s-})\tilde\pi(dz,ds)
=\int_0^t\int_ZJ^{h_s(z)}\phi(X_{s-})\,\tilde\pi(dz,ds)
+\int_0^t\int_ZD_i\phi(X_{s-})h^i_s(z)\,\tilde\pi(dz,ds), 
$$ 
and by virtue of \eqref{Jintegral1} and \eqref{Jintegral2}, 
$$
\int_0^t\int_ZJ^{h_s(z)}\phi(X_{s-})\,\tilde\pi(dz,ds)
+\int_0^t\int_ZJ^{h_s(z)}\phi(X_{s-})\, \mu(dz)\,ds
=\int_0^t\int_ZJ^{h_s(z)}\phi(X_{s-})\,\pi(dz,ds). 
$$
Hence
$$
\int_0^t\int_ZI^{h_s(z)}\phi(X_{s-})\,\tilde\pi(dz,ds)+
\int_0^t\int_ZJ^{h_s(z)}\phi(X_{s-})\, \mu(dz)\,ds
$$
$$
=\int_0^t\int_ZD_i\phi(X_{s-})h^i_s(z)\,\tilde\pi(dz,ds)
+\int_0^t\int_ZJ^{h_s(z)}\phi(X_{s-})\,\pi(dz,ds), 
$$
which shows that Theorem \ref{theorem0} holds 
under the additional condition that $|h|$ is bounded. 
To prove the theorem in full generality we approximate 
$h$ by $h^{(n)}=(h^{1(n)},...,h^{M(n)})$,  where 
$h_t^{in}=-n\vee h_t^{i}\wedge n$ for integers $n\geq1$, and define 
$$
X^{(n)}_t:=X_0+\int_0^tf_s\,ds+\int_0^tg_s^{r}\,dw_s^r
+\int_0^t\int_Z\bar h_s(z)\,\pi(dz,ds)
+\int_0^t\int_Zh^{(n)}_s(z)\,\tilde{\pi}(dz,ds), \quad t\in[0,T]. 
\quad
$$
Clearly, for all $(\omega,t,z)$ 
\begin{equation*}                                           
|h^{(n)}|\leq \min(|h|, nM)\quad
\text{and \quad $h^{(n)}\rightarrow h$\quad 
as $n\rightarrow \infty$}. 
\end{equation*}
Therefore Theorem \ref{theorem0} for $X^{(n)}$ holds, and 
$$
\lim_{n\to\infty}\int_0^T\int_Z|h^{(n)}_t(z)-h_t(z)|^2\,\mu(dz)\,dt=0\,\,(\rm{a.s.}), 
$$
which implies 
$$
\sup_{t\leq T}|X^{(n)}_t-X_t|\to 0 \quad\text{in probability as $n\to\infty$.}
$$
Thus there is a strictly increasing subsequence 
of positive integers $(n_k)_{k=1}^{\infty}$  such that 
$$
\lim_{k\to\infty}\sup_{t\leq T}|X^{(n_k)}_t-X_t|=0\quad(\rm{a.s.}), 
$$ 
which implies  
$$
\rho:=\sup_{k\geq1}\sup_{t\leq T}|X^{(n_k)}_t|<\infty\quad (\rm{a.s.}). 
$$
Hence it is easy to pass to the limit $k\to\infty$ in $\phi(X_t^{(n_k)})$ 
and in the first two integral terms in the equation for $\phi(X_t^{(n_k)})$ in 
Theorem \ref{theorem0}. To pass to the limit in the other terms in this equation  
notice that since $\pi(dz,dt)$ is a counting measure of a point process, 
from the condition 
for $\bar h$  in \eqref{integrals} we get 
\begin{equation}                                                                    \label{esssup1}
\xi:={\esssup}\,|\bar h|<\infty\,\,(\rm{a.s.}),  
\end{equation}
where ${\esssup}$ denotes the essential supremum operator 
with respect to the measure $\pi(dz,dt)$ over $Z\times[0,T]$. 
Similarly, from the condition for $h$ we have
\begin{equation}                                                                     \label{esssup2}
\eta:={\esssup}\,|h|<\infty\,\,(\rm{a.s.}).   
\end{equation}
This can be seen by noting that 
for the sequence 
of predictable stopping times 
$$
\tau_j=\inf\left\{t\in[0,T]:\int_0^t\int_Z|h_s(z)|^2\,\mu(dz)\,ds\geq j \right\}, 
\quad\text{$j=1,2,...$}, 
$$
we have 
$$
E\int_0^T\int_Z{\bf1}_{t\leq\tau_j}|h_t(z)|^2\,\pi(dz,dt)
=E\int_0^T\int_Z{\bf1}_{t\leq\tau_j}|h_t(z)|^2\,\mu(dz)\,dt\leq j<\infty, 
$$
which gives 
$$
\int_0^T\int_Z|h_t(z)|^2\,\pi(dz,dt)<\infty 
\quad
\text{almost surely on 
$\Omega_j=\{\omega\in\Omega:\tau_j\geq T\}$ for each $j\geq1$.} 
$$
Since $(\tau_j)_{j=1}^{\infty}$ is an increasing sequence 
converging to infinity, we have $P(\cup_{j=1}^{\infty}\Omega_j)=1$, i.e., 
\begin{equation}                                                                      \label{pi}
\int_0^T\int_Zh^2_t(z)\,\pi(dz,dt)<\infty\,\,(\rm{a.s.}),    
\end{equation}
which implies \eqref{esssup2}. 
 By \eqref{esssup1} and 
the first inequality 
in \eqref{Taylor}, we have  
$$
|I^{\bar h_t(z)}\phi(X_{t-}^{(n_k)})|+|I^{\bar h_t(z)}\phi(X_{t-})|
\leq 2\sup_{|x|\leq \rho+\xi}|D\phi(x)||\bar h_t(z)|<\infty
$$
almost surely for $\pi(dz,dt)$-almost every $(z,t)\in Z\times[0,T]$. 
Hence by Lebesgue's theorem on dominated convergence we get 
$$
\lim_{k\to\infty}\int_0^T\int_Z
|I^{\bar h_s(z)}\phi(X^{(n_k)}_{s-})-I^{\bar h_s(z)}\phi(X_{s-})|\,\pi(dz,ds)=0 \quad \rm{(a.s.)},  
$$
which implies that for $k\to\infty$ 
$$
\int_0^t\int_Z
I^{\bar h_s(z)}\phi(X^{(n_k)}_{s-})\,\pi(dz,ds)
\to
\int_0^t\int_Z I^{\bar h_s(z)}\phi(X_{s-})\,\pi(dz,ds)
$$
almost surely, uniformly in $t\in[0,T]$. Clearly, 
$$
|D_i\phi(X^{(n_k)}_{t-})h^{i(n_k)}_t(z)|^2+|D_i\phi(X_{t-})h^{i}_t(z)|^2
\leq 2\sup_{|x|\leq \rho}|D\phi(x)|^2|h_t(z)|^2
$$
almost surely for all $(z,t)\in Z\times[0,T]$. Hence by Lebesgue's theorem 
on dominated convergence 
$$
\lim_{k\to\infty}\int_0^T\int_Z
|D_i\phi(X^{(n_k)}_{t-})h^{i(n_k)}_t(z)-D_i\phi(X_{t-})h^{i}_t(z)|^2
\,\mu(dz)\,dt=0\quad \rm{(a.s.)},  
$$
which implies that for $k\to\infty$ 
$$
\int_0^t\int_Z
D_i\phi(X^{(n_k)}_{s-})h^{i(n_k)}_s(z)\,\tilde\pi(dz,ds)
\to
\int_0^t\int_Z
D_i\phi(X_{s-})h^{i}_s(z)\,\tilde\pi(dz,ds)   
$$
in probability, uniformly in $t\in[0,T]$. 
Finally note that by using the second inequality in \eqref{Taylor} 
together with \eqref{esssup2} we have  
$$
|J^{h^{(n_k)}_t(z)}\phi(X_{t-}^{(n_k)})|+|J^{h_t(z)}\phi(X_{t-})|
\leq 2\sup_{|x|\leq \rho+\eta}|D^2\phi(x)||h_t(z)|^2
$$
almost surely for $\pi(dz,dt)$-almost every $(z,t)\in Z\times[0,T]$. Hence, 
taking into account \eqref{pi}, by Lebesgue's theorem 
on dominated convergence we obtain  
$$
\lim_{k\to\infty}\int_0^T\int_Z
|J^{h^{(n_k)}_t(z)}\phi(X^{(n_k)}_{t-})-J^{h_t(z)}\phi(X_{t-})|
\,\pi(dz,dt)=0\quad \rm{(a.s.)},  
$$
which implies that for $k\to\infty$ 
$$
\int_0^t\int_ZJ^{h^{(n_k)}_s(z)}\phi(X^{(n_k)}_{s-})\,\pi(dz,ds)
\to
\int_0^t\int_ZJ^{h_s(z)}\phi(X_{s-})\,\pi(dz,ds)
$$
almost surely, uniformly in $t\in[0,T]$ for every $T>0$, that 
finishes the proof of the theorem.  
\end{proof}
\begin{remark}
One can give a different proof of Theorem \ref{theorem0} 
by showing that for finite measures $\mu^k$, 
the It\^o formula 
for general semimartingales, Theorem VIII.27 in \cite{DM1982}, 
 applied to $(X_t)_{t\geq0}$, can be rewritten as 
equation \eqref{Ito01}. Hence by an approximation procedure 
one can get the general case of $\sigma$-finite measures $\mu^k$. 
\end{remark}

\begin{corollary}                                                        \label{corollary Ito}
Let conditions \eqref{integrands} and \eqref{integrals} hold. Then for 
any $p\geq2$ the process $|X_t|^p$ is a semimartingale such that  
\begin{align}
|X_t|^p
= &|\psi|^p+ p\int_0^t|X_s|^{p-2}X^i_sg_s^{ir}\,dw_s^r                           \nonumber\\
& + \tfrac{p}{2}\int_0^t\left(2|X_s|^{p-2}X^i_sf^i_s+
(p-2)|X_s|^{p-4}|X^i_sg^{i\cdot}_s|_{l_2}^2
+\sum_{i=1}^M|X_s|^{p-2}|g^{i\cdot}_s|_{l_2}^2\right)\,ds                                                   \nonumber\\
& + \sum_{k=1}^mp\int_0^t\int_{Z_k}|X_{s-}|^{p-2}X^i_{s-}h^{ik}_s(z)\,\tilde\pi^k(dz,ds)
\nonumber\\
&
+\sum_{k=1}^m\int_0^t\int_{Z_k}(|X_{s-}+\bar{h}^k_s|^p-|X_{s-}|^p)\,\pi^k(dz,ds)                \nonumber\\
&+\sum_{k=1}^m\int_0^t\int_{Z_k}\left(
|X_{s-}+h^k_s|^p-|X_{s-}|^p-p|X_{s-}|^{p-2}X^i_{s-}h^{ik}_s\right)\,\pi^k(dz,ds)                \label{Itop}
\end{align}
almost surely for all $t\geq0$, where, and through the paper, the convention $0/0:=0$ is used 
whenever it occurs.  
\end{corollary}
\begin{proof}
Notice that $\phi(x)=|x|^p$ for $p\geq2$ belongs to $C^2(\bR^M)$  
with 
$$
D_i|x|^p=p|x|^{p-2}x^i, \quad D_jD_i|x|^p=p(p-2)|x|^{p-4}x^ix^j+p|x|^{p-2}\delta_{ij},  
$$
where $\delta_{ij}=1$ for $i=j$ and $\delta_{ij}=0$ for $i\neq j$. 
Hence it is easy to see that Theorem \ref{theorem0} applied to $\phi(x)=|x|^p$ gives the corollary. 
\end{proof} 
 
The above corollary will be used to obtain an It\^o's
formulas for jump processes in $L_p$-spaces presented in the next section.

\mysection{It\^o formula in $L_p$ spaces}                         \label{section L_p Ito}
 
It\^o formulas in infinite dimensional spaces play important roles 
in studying stochastic PDEs. Our theorem below is motivated by applications 
in the theory of stochastic integro-differential equations arising in nonlinear filtering 
theory of jump diffusions. To present it first we need to introduce some notations, where 
$T$ is a fixed positive number, and $d\geq1$ and $M\geq1$ are fixed integers.  

The Borel $\sigma$-algebra of a topological space $V$ is denoted by $\cB(V)$. 
For $p, q\geq 1$ we denote by $L_p=L_p(\bR^d,\bR^M)$  
and $\cL_{q}=\cL_{q}(Z, \bR^M)$ 
the Banach spaces of $\bR^M$-valued Borel-measurable functions of $f=(f^i(x))_{i=1}^M$ 
and $\cZ$-measurable functions $h=(h^{i}(z))_{i=1}^M$ of $x\in\bR^d$ and $z\in Z$, respectively 
such that 
$$
|f|_{L_p}^p=\int_{\bR^d}|f(x)|^p\,dx<\infty
\quad
\text{and}
\quad 
|h|^{q}_{\cL_{q}}=\int_{Z}|h(z)|^{q}\,\mu(dz)<\infty. 
$$
The notation 
$\cL_{p,q}$ means the space $\cL_{p}\cap\cL_{q}$ with the norm 
$$
|v|_{\cL_{p,q}}=\max(|v|_{\cL_{p}},|v|_{\cL_{q}}) \quad\text{for $v\in \cL_{p}\cap\cL_{q}$}.
$$
As usual, $W^1_p$ denotes the space of functions $u\in L_p$ such that 
$D_iu\in L_p$ for every $i=1,2,...,d$, where $D_iv$ means the generalised derivative 
of $v$ in $x^i$ for locally integrable functions $v$ on $\bR^d$. The norm of $u\in W^1_p$ 
is defined by 
$$
|u|_{W^1_p}=|u|_{L_p}+\sum_{i=1}^d|D_iu|_{L_p}. 
$$
We use the notation $L_p=L_p(\ell_2)$ for $L_p(\bR^d,\ell_2)$, 
the space of Borel-measurable 
functions $g=(g^{ir})$ on $\bR^d$ with values in $\ell_2$ such that 
$$
|g|_{L_p}^p=\int_{\bR^d}|g(x)|_{\ell_2}^p\,dx<\infty.  
$$
For $p,q\in [0,\infty)$ we denote by $L_{p}=L_p(\cL_{p,q})$ and $L_p=L_p(\cL_q)$ 
the Banach spaces of Borel-measurable functions 
$h=(h^i(x,z))$ and $\tilde{h}=(\tilde{h}^i(x,z))$ of $x\in\bR^d$ 
with values in $\cL_{p,q}$ and $\cL_q$,  
respectively,  
such that 
$$
|h|^p_{L_{p}}=\int_{\bR^d}|h(x,\cdot)|^p_{\cL_{p,q}}\,dx<\infty
\quad\text{and}\quad  
|\tilde{h}|^p_{L_{p}}=\int_{\bR^d}|\tilde{h}(x,\cdot)|^p_{\cL_q}\,dx<\infty. 
$$
For $p\geq2$ and a separable real Banach 
space $V$ we denote by $\bL_p=\bL_p(V)$ 
the space of predictable $V$-valued functions $f=(f_t)$ of 
$(\omega,t)\in\Omega\times[0,T]$ such that 
$$
|f|_{\bL_p}^p=E\int_0^T|f_t|^p_{V}\,dt<\infty. 
$$
In the sequel $V$ will be $L_p(\bR^d,\bR^M)$, $L_p(\bR^d,\ell_2)$, or 
$L_p(\bR^d,\cL_{p,2})$. When $V=L_p(\bR^d,\cL_{p,2})$ then for $\bL_p(V)$  
the notation $\bL_{p,2}$ is also used.  For $\varepsilon\in(0,1)$ 
and locally integrable functions $v$ of $x\in\bR^d$ we use the notation  
$v^{(\varepsilon)}$ for the mollifications of $v$, 
\begin{equation}                                          \label{mollification def}
v^{(\varepsilon)}(x)=\int_{\bR^d}v(x-y)k_{\varepsilon}(y)\,dy, \quad x\in\bR^d, 
\end{equation}
where $k_{\varepsilon}(y)=\varepsilon^{-d}k(y/\varepsilon)$ for $y\in\bR^d$  
with a fixed function $k\in C_0^{\infty}$ of unit integral. 
If $v$ is a locally Bochner-integrable function on $\bR^d$,  
taking values in a Banach space, then the mollification of $v$ is defined 
as \eqref{mollification def} in the sense of Bochner integral.

Recall that the summation convention with respect 
to integer valued indices is used throughout 
the paper.

\begin{assumption}                                                          \label{assumption1}
Let $\psi^i$ be an $L_p(\bR^d,\bR)$-valued $\cF_0$-measurable random variable,  
$(u^{i}_{t})_{t\in0,T}$ be a progressively measurable $L_p$-valued 
process and let $f^{i\alpha}$, 
$g^i=(g^{ir})_{r=1}^{\infty}$ and $h^i$ be predictable functions on 
$\Omega\times[0,T]\times Z$ with values in $L_p(\bR^d,\bR)$, $L_p(\bR^d,l_2)$ 
and $L_p(\bR^d,\cL_{p,2})$, respectively, for each $i=1,2,...,M$ and $\alpha=0,1,...,d$,  
such that the following conditions are satisfied for each $i=1,2,...,M$. 
\begin{enumerate}
\item[(i)] We have $u^i_t\in W^1_p$ for $P\otimes dt$-a.e. 
$(\omega,t)\in\Omega\times[0,T]$ such that 
\begin{equation}                                                                      \label{1}
\int_0^T|u_t^i|^p_{W^1_p}\,dt<\infty\quad\rm{(a.s.)}.
\end{equation}
\item[(ii)] Almost surely 
\begin{equation}                                                                      \label{2}
\cK^p_{p}(T):=\sum_{i=1}^M\int_0^T\int_{\bR^d}\sum_{\alpha}|f^{i\alpha}_t(x)|^p
+|g^i_t(x)|_{l_2}^p+|h^i_t(x)|^p_{\cL_{p,2}}\,dx\,dt<\infty
\end{equation}
\item[(iii)] 
For every $\varphi\in C_0^{\infty}(\bR^d)$ we have 
\begin{equation}                                                                \label{eq1}
(u^i_t,\varphi)
=(\psi,\varphi)
+\int_0^t(f^{i\alpha}_s,D^{\ast}_{\alpha}\varphi)\,ds
+\int_0^t(g_s^{ir},\varphi)\,dw_s^r
+\int_0^t\int_Z(h^i_s(z),\varphi)\,\tilde{\pi}(dz,ds)
\end{equation}
for $P\otimes dt$-almost every $(\omega,t)\in\Omega\times[0,T]$.
\end{enumerate}
\end{assumption}
In equation \eqref{eq1}, and later on, we use the notation $(v,\phi)$ 
for the Lebesgue integral over $\bR^d$ of the product $v\phi$ 
for functions  $v$ and $\phi$ on $\bR^d$ when their product 
and its integral are well-defined. Below $u$ stands for $(u^1,...,u^M)$. 

\begin{theorem}                                                          \label{theorem1}
Let Assumption \ref{assumption1} hold with $p\geq2$. Then there is 
an $L_p(\bR^d,\bR^M)$-valued adapted cadlag process
 $\bar u=(\bar u^i_t)_{t\in[0,T]}$ 
such that equation \eqref{eq1}, with $\bar u$ in place of $u$, holds 
for each $\varphi\in C_0^{\infty}(\bR^d)$  almost surely for all $t\in[0,T]$. 
Moreover,  $u=\bar u$ for $P\otimes dt$-almost every 
$(\omega,t)\in\Omega\times[0,T]$, and almost surely 
\begin{align}
|\bar u_t|^p_{L_p}
&= |\psi|_{L_p}^p
+p\int_0^t\int_{\mathbb{R}^d}|\bar u_s|^{p-2}\bar u^i_sg^{ir}_s\,dx\, dw^r_s                    \nonumber\\
& 
+\frac{p}{2}\int_0^t\int_{\mathbb{R}^d} 2|u_s|^{p-2}\bar u^i_sf^{i0}_s
-2|u_s|^{p-2}D_k u^i_sf_s^{ik}-(p-2)|u_s|^{p-4}u^i_sf^{ik}_sD_k|u_s|^2\,dx\,ds     \nonumber\\
&
+\frac{p}{2}\int_0^t\int_{\mathbb{R}^d}
(p-2)|u_s|^{p-4}
 |u^i_sg_s^{ i\cdot}|_{l_2}^2
+|u_s|^{p-2}
\sum_{i=1}^M|g^{i\cdot}_s|^2_{l_2}\,dx\,ds 
                                                 \nonumber\\
& 
+p\int_0^t\int_Z\int_{\mathbb{R}^d}|u_{s-}|^{p-2}u^i_{s-}h^i_s
\,dx\,\tilde{\pi}(dz,ds)                                                                                              \nonumber\\
&
+\int_0^t\int_Z\int_{\mathbb{R}^d}
(|\bar u_{s-}+h_s|^p-|\bar u_{s-}|^p-p|\bar u_{s-}|^{p-2}\bar u^i_{s-}h^i_s)
\,dx\,\pi(dz,ds)                                                                                                         \label{Ito1}
\end{align}
for all $t\in[0,T]$, where   $\bar u_{s-}$ means the left-hand limit 
in $L_p$  of $\bar u$ at $s$. If $f^i=0$ for $i=1,2,...,d$ then the above statements 
hold if Assumption \ref{assumption1} is satisfied with (i) replaced in it 
with the weaker condition that 
\begin{equation}                                                                    \label{L}
\int_0^T|u^i_t|^p_{L_p}\,dt<\infty\quad\rm{(a.s.)}.
\end{equation}
\end{theorem}

Notice that for $M=1$ equation \eqref{Ito1} has the simpler form 
$$ 
|\bar u_t|^p_{L_p}= 
 |\psi|_{L_p}^p+p\int_0^t\int_{\mathbb{R}^d}|u_s|^{p-2}u_sg_s^r\,dx\,dw^r_s
 $$
 $$
+\frac{p}{2}\int_0^t\int_{\mathbb{R}^d}
\big(2|u_s|^{p-2}u_sf^0_s-2(p-1)|u_s|^{p-2}f^i_sD_iu_s
+(p-1)|u_s|^{p-2}|g_s|^2_{l_2}   \big)\,dx\,ds
$$
$$
+p\int_0^t\int_Z\int_{\mathbb{R}^d}
|\bar u_{s-}|^{p-2}\bar u_{s-}h_s\,dx\,\tilde{\pi}(dz,ds)
$$
\begin{equation}                                                                          \label{Ito2}
+\int_0^t\int_Z\int_{\mathbb{R}^d}\big( | \bar u_{s-}
+h_s|^p-| \bar u_{s-}|^p-p| \bar u_{s-}|^{p-2}\bar u_{s-}h_s  \big)\,dx\,\pi(dz,ds).    
\end{equation}
\medskip

Theorem \ref{theorem1} generalises Theorem 2.1 from \cite{K2010}, 
and we use ideas and methods from 
\cite{K2010} to prove it.  The basic idea in \cite{K2010} adapted to our situation 
can be explained as follows. Assume first that $f^{i\alpha}=0$ for 
$\alpha=1,2,...,d$, and suppose from \eqref{eq1} we could show the existence of a random field 
$\bar u=\bar u(t,x)$ and 
suitable modifications of the integrals of $f^{i}:=f^{i0}_s(x)$, $g=g^{ir}_s(x)$ 
and $h^i_s(x,z)$ against $ds$, $dw^r_s$ 
and $\tilde\pi(dz,ds)$, respectively, satisfying appropriate measurability 
conditions such that the equation 
\begin{equation}                                      \label{eqx}
\bar u^i_t(x)
=\psi^{i}(x)
+\int_0^tf^{i}_s(x)\,ds
+\int_0^tg_s^{ir}(x)\,dw_s^r
+\int_0^t\int_Zh^i_s(x,z)\,\tilde{\pi}(dz,ds)
\end{equation}
holds for every $x\in\bR^d$ and $i=1,2,...,M$. 
Then applying It\^o's formula \eqref{Itop} from Corollary 
\ref{corollary Ito} to 
$|\bar u_t(x)|^p=(\sum_{i}|\bar u^i_t(x)|^2)^{p/2}$ for every $x\in\bR^d$, then 
integrating over $\bR^d$, and finally using suitable stochastic Fubini theorems, we could  
obtain \eqref{Ito1} when $f^{i\alpha}=0$ for $\alpha\geq1$. When 
$f^{i \alpha}\neq0$ we could  
take 
$$
u^{i(\varepsilon)},\quad\psi^{i(\varepsilon)},\quad 
f^{i(\varepsilon)}:=f^{i0(\varepsilon)}+\sum_{k=1}^dD_kf^{ik(\varepsilon)},
\quad g^{ir(\varepsilon)}
\quad \text{and\quad$h^{i(\varepsilon)}$}
$$
instead of $u^i$, $\psi^i$, $f^{i}$, $g^{ir}$ and $h^{i}$ above, respectively to apply 
the theorem in the special case, and let 
$\varepsilon\to0$ in the corresponding It\^o formula after integrating by parts in the terms containing 
$D_kf^{ik(\varepsilon)}$ for $k=1$,....,$d$. 
Notice that we can formally obtain 
equation \eqref{eqx} from 
\eqref{eq1} with $f^{i1}=...=f^{id}=0$ and a suitable process $\bar u$ in place of $u$, 
by substituting $\delta_x$, the Dirac delta at $x$, 
in place of $\varphi$. Clearly, we cannot substitute $\delta_x$, but we can substitute 
approximations $k_{\varepsilon}(x-\cdot)$ of it 
to get 
\begin{equation}                                                   \label{eqex}
\bar u^{i(\varepsilon)}_t(x)
=\psi^{i(\varepsilon)}(x)
+\int_0^tf^{i(\varepsilon)}_s(x)\,ds
+\int_0^tg_s^{ir(\varepsilon)}(x)\,dw_s^r
+\int_0^t\int_Zh^{i(\varepsilon)}_s(x,z)\,\tilde{\pi}(dz,ds)
\end{equation}
in place of \eqref{eqx}. Therefore the above 
strategy is modified as follows.  One chooses suitable representative 
of the stochastic integrals in \eqref{eqex} so that one could apply It\^o's formula 
\eqref{Itop} to $|\bar u^{(\varepsilon)}_t(x)|^p$ for each $x\in\bR^d$,  
integrate the obtained formula over $\bR^d$, then interchange the order of the integrals,  
and finally let $\varepsilon\to0$ to prove equation \eqref{Ito1} when $f^{ik}=0$ 
for $i=1,2,...,M$ and $k=1,2,...,d$. 

To implement the above idea we fix a $p\geq2$ and introduce a class of functions $\cU_p$,   
the counterpart of the class  $\cU_p$  given in \cite{K2010}.  
Let $\mathcal{U}_p$ denote the set of $\bR^M$-valued functions 
$u=u_t(x)=u_t(\omega,x)$ on $\Omega\times [0,T]\times \mathbb{R}^d$ such that
\begin{enumerate}[(i)]
\item $u$ is $\mathcal{F}\otimes\mathcal{B}([0,T])\otimes\mathcal{B}(\mathbb{R}^d)$-measurable,
\item for each $x\in\mathbb{R}^d$, $u_t(x)$ is $\mathcal{F}_t$-adapted,
\item $u_t(x)$ is cadlag in $t\in[0,T]$ for each $(\omega,x)$,
\item $u_t(\omega,\cdot)$ as a function of 
$(\omega,t)$ is $L_p$-valued, $\mathcal{F}_t$-adapted and cadlag in $t$ 
for every $\omega\in \Omega$.
\end{enumerate}
The following lemmas present suitable versions of Lebesgue and It\^o integrals 
with values in $L_p$. The first two of them are obvious corollaries 
of Lemmas 4.3 and 4.4 in \cite{K2010}. 

\begin{lemma}                                         \label{lemma f}
Let $f\in\bL_p(V)$ for $V=L_p(\bR^d,\bR^M)$. 
Then there exists a function $m\in \cU_p$ 
such that for each $\varphi\in C_0^\infty$ almost surely
$$
(m_t,\varphi)=\int_0^t(f_s,\varphi)\,ds
$$
holds for all $t\in[0,T]$. Furthermore, we have
$$
E\int_{\bR^d}\sup_{t\leq T}|m_t(x)|^p\,dx\leq NT^{p-1}E\int_0^T|f_s|^p_{L_p}\,ds,
$$
with a constant $N=N(p,M)$. 
\end{lemma}

\begin{lemma}                                                                   \label{lemma g}
Let $g$ be from $\bL_p(V)$ for $V=L_p(\bR^d,\ell_2)$. 
Then there exists a function $a\in\cU_p$ 
such that for each $\varphi\in C_0^\infty$ almost surely 
$$
(a_t,\varphi)=\sum_{r=1}^\infty \int_0^t(g_s^r,\varphi)\,dw_s^r
$$
holds for all $t\in[0,T]$. Furthermore, we have
$$
E\int_{\bR^d}\sup_{t\leq T}|a_t(x)|^p\,dx
\leq NT^{(p-2)/2}E\int_0^T|g_s|^p_{L_p}\,ds
$$
with a constant  $N=N(p,M)$.
\end{lemma}

The proof of the following lemma can be found in \cite{GW Ito}.
\begin{lemma}                                                                       \label{lemma h}
Let $h\in\mathbb{L}_{p,2}$. 
Then there exists a function $b\in\mathcal{U}_p$ 
such that for each real-valued $\varphi\in L_q(\bR^d)$ 
with $q=p/(p-1)$, almost surely 
\begin{equation}                                                                   \label{weak3}
(b_t,\varphi)=\int_0^t\int_Z(h_s,\varphi)\,\tilde{\pi}(dz,ds)
\end{equation}
for all $t\in[0,T]$,   and 
\begin{equation}                                                                   \label{weaksup}
E\sup_{t\leq T}|(b_t,\varphi)|
\leq 3T^{(p-2)/(2p)}|\varphi|_{L_q}\left(E\int_0^T|h_t|_{L_p(\cL_2)}^p\,dt\right)^{1/p}. 
\end{equation}
Furthermore
\begin{equation}                                                                       \label{Lp}
E\int_{\bR^d}\sup_{t\leq T}|b_t(x)|^p\,dx
\leq NE\int_0^T|h_t|^p_{L_p(\cL_p)}\,dt+NT^{(p-2)/2}E\int_0^T|h_t|^p_{L_p(\cL_2)}\,dt
\leq N'|h|^p_{\bL_{p,2}}
\end{equation}
with constants $N=N(p,M)$ and $N'=N'(p,M,T)$.
\end{lemma}
We are now in the position to sketch the proof of Theorem \ref{theorem1}. 
Technical details can be found in \cite{GW Ito}.

\begin{proof}[Proof of Theorem \ref{theorem1}(Sketch)] 
By using standard stopping time arguments we may assume 
$E|\psi^i|^p_{L_p}<\infty$ and that 
$$
E\int_0^T|u_t^i|^p_{W^1_p}\,dt<\infty, \quad E\cK_p^p(T)<\infty
\quad\text{and\quad $E\int_0^T|u_t^i|^p_{L_p}\,dt<\infty$} 
$$
hold in place of \eqref{1}, \eqref{2} and \eqref{L}, respectively for every 
$i=1,2,...,M$. 
We prove first the last sentence of the theorem. We have 
$f^{ik}=0$ for $i=1,2,...,M$, $k=1,2,...,d$ and use the notation 
$f^i:=f^{i0}$.  
By Lemmas \ref{lemma f}, \ref{lemma g} and \ref{lemma h} there exist $a=(a^i)$ 
and $b=(b^i)$ and $m=(m^i)$ in $\mathcal{U}_p$ 
such that for each $\varphi\in C_0^\infty$ almost surely
$$
(a_t^i,\varphi)=\int_0^t(f^i_s,\varphi)ds,
\quad 
(b_t^i,\varphi)=\int_0^t(g^{ir}_s,\varphi)dw_s^r
$$
and
$$
(m^i_t,\varphi)=\int_0^t\int_Z(h^i_s,\varphi)\,\tilde{\pi}(dz,ds)
$$
for all $t\in [0,T]$ and $i=1,...,M$. Thus $a+b+m$ is an 
$L_p$-valued adapted cadlag process such that 
for $\bar u_t:=\psi+a_t+b_t+m_t$ we have 
$(\bar u_t,\varphi)=
(u_t,\varphi)$ for each $\varphi\in C_0^{\infty}$ 
for $P\otimes dt$ almost every 
$(\omega,t)\in\Omega\times[0,T]$. 
Hence, by taking a countable set 
$\Phi\subset C_0^{\infty}$ such that $\Phi$ is dense in $L_q$, 
we get that  $\bar u=u$ 
for $P\otimes dt$ almost everywhere 
as $L_p$-valued functions. Moreover, for each 
$\varphi\in C_0^{\infty}$
\begin{equation}                                                          \label{equ1}
(\bar u^i_t,\varphi)
=(\psi,\varphi)
+\int_0^t(f^i_s,\varphi)\,ds
+\int_0^t(g_s^{ir},\varphi)\,dw_s^r
+\int_0^t\int_Z(h^i_s(z),\varphi)\,\tilde{\pi}(dz,ds)
\end{equation}
almost surely for all $t\in[0,T]$, $i=1,2,...,M$, 
since on both sides we have cadlag 
processes.  
By the estimates of Lemmas 
\ref{lemma f}, \ref{lemma g} and \ref{lemma h},
$$
E\int_{\mathbb{R}^d}\sup_{t\leq T}|u_t(x)|^p\,dx
$$
\begin{equation}                                            \label{u sup estimate}
\leq N\left( E|\psi|^p_{L_p}+ |f|^p_{\bL_p}
+|g_s|^p_{\bL_p}
+|h|^p_{\bL_{p,2}} \right)<\infty, 
\end{equation}
where $N=N(p,M, T)$ is a constant.  
Substituting $k_{\varepsilon}(x-\cdot)$ 
in place of $\varphi$  in equation \eqref{equ1}, 
for $\varepsilon>0$ and $x\in\bR^d$ 
we have \eqref{eqex} 
almost surely for all $t\in[0,T]$ for $i=1,2,...,M$. Hence 
by Corollary \ref{corollary Ito}
for each $x\in \bR^d$ we have almost surely 
$$
|\bar u^{(\varepsilon)}_t(x)|^p
=  |\psi^{(\varepsilon)}(x)|^p
+\int_0^t
p|\bar u^{(\varepsilon)}_{s-}(x)|^{p-2}
\bar u^{i(\varepsilon)}_{s-}(x)g^{ir(\varepsilon)}_s(x)\,dw^r_s 
$$
$$
+\int_0^t
p|\bar u_{s-}^{(\varepsilon)}(x)|^{p-2}
\bar u_{s-}^{(\varepsilon)i}f_s^{i(\varepsilon)}(x) \,ds 
$$
$$             
+\tfrac{p}{2}\int_0^t\big((p-2)
|\bar u_{s-}^{(\varepsilon)}(x)|^{p-4}
|\bar u_{s-}^{i(\varepsilon)}(x)g^{i\cdot(\varepsilon)}_s(x)|^2_{l_2}
+
|\bar u_{s-}^{(\varepsilon)}(x)|^{p-2}
|g_s^{(\varepsilon)}(x)|_{\ell_2}^2  \big)\,ds      
$$
\begin{equation}                                                                            \label{Itox}
+\int_0^t
\int_Z p|\bar u_{s-}^{(\varepsilon)}(x)|^{p-2}
\bar u_{s-}^{(\varepsilon)i}(x)h_s^{(\varepsilon)i}(x)
\,\tilde{\pi}(dz,ds)
+\int_0^t\int_Z 
J^{h_s^{(\varepsilon)}(x,z)}|\bar u_{s-}^{(\varepsilon)}(x)|^p \,\pi(dz,ds),        
\end{equation}
for all $t\in[0,T]$, where the notation 
$$
J^a|v|^p=|v+a|^p-|v|^p-a^iD_{i}|v|^p=|v+a|^p-|v|^p-pa^i|v|^{p-2}v^i
$$
is used for vectors $a=(a^1,...,a^M):=\bar u_{s-}^{(\varepsilon)}(x)$ and 
$(v^1,...,v^M):=h_s^{(\varepsilon)}(x,z)\in\bR^M$. 
Furthermore, integrating \eqref{Itox} over $\bR^d$ and using deterministic and stochastic 
Fubini theorems, see in \cite{GW Ito}, we get  
$$
|\bar u^{(\varepsilon)}_{t}|_{L_p}^p= 
 |\psi^{(\varepsilon)}|_{L_p}^p
+\int_0^t\int_{\bR^d}
p|u^{(\varepsilon)}_s|^{p-2}u^{i(\varepsilon)}_sg^{ir(\varepsilon)}_s\,dx\,dw^r_s
$$
$$
+\tfrac{p}{2}\int_0^t\int_{\bR^d}
2|u_s^{(\varepsilon)}|^{p-2}u_s^{i(\varepsilon)}f_s^{i(\varepsilon)}               
+
(p-2)
|u_s^{(\varepsilon)}|^{p-4}|u_s^{i(\varepsilon)}g^{i\cdot(\varepsilon)}_s|_{l_2}^2
+|u_s^{(\varepsilon)}|^{p-2}|g_s^{(\varepsilon)}|_{l_2}^2
\,dx\,ds                                                                                                                    
$$
\begin{equation}                                                                    \label{Itoe}
+\int_0^t\int_Z\int_{\bR^d} 
p|u_{s-}^{(\varepsilon)}|^{p-2}
u_{s-}^{i(\varepsilon)}h_s^{i(\varepsilon)}\,dx\,\tilde{\pi}(dz,ds)
+\int_0^t\int_Z\int_{\bR^d} 
J^{h^{(\varepsilon)}_s}|u_{s-}^{(\varepsilon)}|^p\,dx \,\pi(dz,ds)                                      
\end{equation}
almost surely for all $t\in[0,T]$. 
Finally, by taking $\varepsilon\to0$ in \eqref{Itoe}, we obtain \eqref{Ito1} with 
$f^{ik}=0$ for $i=1,2,...,M$ and $k=1,2,...,d$. 

Let us prove now the other statements of the theorem.  By 
taking $\varphi^{(\varepsilon)}$ in place of $\varphi$ in equation \eqref{eq1}  
we get 
\begin{equation}                                                                    \label{differential}
(u_t^{i(\varepsilon)},\varphi)=(\psi^{i(\varepsilon)},\varphi)
+\int_0^t(f_s^{(i\varepsilon)},\varphi)\,ds+\int_0^t(g_s^{ir(\varepsilon)},\varphi)\,dw_s^r
+\int_0^t\int_Z(h_s^{i(\varepsilon)},\varphi)\,\tilde{\pi}(dz,ds)
\end{equation}
for $P\otimes dt$ almost every $(\omega,t)\in\Omega\times[0,T]$ for each $\varphi\in C_0^\infty$, 
$i=1,2,...,m$, where
$$
f_s^{i(\varepsilon)}:=\sum_{k=1}^dD_kf_s^{ik(\varepsilon)}+f_s^{i0(\varepsilon)}, 
\quad i=1,2,...,M,\quad k=1,2,....,d.
$$ 
Hence by virtue of what we have proved above we have an $L_p$-valued adapted cadlag 
process $\bar u^{\varepsilon}=(\bar u^{i\varepsilon})$ such that  
for each $\varphi\in C_0^{\infty}$ almost surely 
\eqref{differential} holds with $\bar u^{i\varepsilon}$ in place of $u^{i(\varepsilon)}$ for all 
$t\in[0,T]$. In particular, for each $\varphi\in C_0^{\infty}$ we have 
$(u^{(\varepsilon)},\varphi)=(\bar u^{\varepsilon},\varphi)$ for $P\otimes dt$-almost 
every $(\omega,t)\in\Omega\times[0,T]$. 
Thus $u^{(\varepsilon)}=\bar u^{\varepsilon}$, as $L_p$-valued functions,  
for $P\otimes dt$-almost every $(\omega,t)\in\Omega\times[0,T]$,  
and almost surely \eqref{Itoe} holds for all $t\in[0,T]$. 
Moreover, using that by integration by parts  
$$
\int_{\mathbb{R}^d}
|u_{s}^{(\varepsilon)}|^{p-2}u_{s}^{i(\varepsilon)}
D_kf_s^{ik(\varepsilon)}\,dx
=-\int_{\mathbb{R}^d}
|u_s^{(\varepsilon)}|^{p-2}f^{ik(\varepsilon)}_sD_ku_s^{i(\varepsilon)}\,dx 
$$
$$
-\tfrac{p-2}{2}\int_{\mathbb{R}^d}
|u_s^{(\varepsilon)}|^{p-4}D_k|u_s^{(\varepsilon)}|^2f^{ik(\varepsilon)}_su_s^{i(\varepsilon)}\,dx    
$$
for $P\otimes dt$-almost every $(\omega,t)\in\Omega\times[0,T]$, we get 
$$
|\bar u_t^{\varepsilon}|^p_{L_p}= 
|\psi^{(\varepsilon)}|_{L_p}^p
+p\int_0^t\int_{\mathbb{R}^d}|\bar u_s^{\varepsilon}|^{p-2}
\bar u_s^{\varepsilon}g_s^{ir(\varepsilon)}\,dx\,dw^r_s                           
$$
$$
+p\int_0^t\int_{\mathbb{R}^d}
|\bar u_s^{\varepsilon}|^{p-2}\bar u_s^{\varepsilon}
f^{0(\varepsilon)}_s-
|\bar u_s^{\varepsilon}|^{p-2}
f_s^{ik(\varepsilon)}D_k u_s^{i(\varepsilon)}                  
\,dx\,ds 
$$
$$
-\tfrac{p}{2}\int_0^t\int_{\mathbb{R}^d}
(p-2)|u_s^{(\varepsilon)}|^{p-4}D_k|u_s^{(\varepsilon)}|^2f^{ik(\varepsilon)}_su_s^{i(\varepsilon)}
\,dx\,ds   
$$
$$
+\int_0^t\int_{\mathbb{R}^d}(p-2)
|\bar u_{s-}^{(\varepsilon)}|^{p-4}
|\bar u_{s-}^{i(\varepsilon)}g^{i\cdot(\varepsilon)}_s|^2_{l_2}
+
|\bar u_{s-}^{(\varepsilon)}|^{p-2}
|g_s^{(\varepsilon)}|_{\ell_2}^2\,dx\,ds
$$
\begin{equation}                                                        \label{Lp Ito formula smooth}
+p\int_0^t\int_Z\int_{\mathbb{R}^d}
|\bar u_{s-}^{\varepsilon}|^{p-2}\bar u_{s-}^{\varepsilon}
h_s^{(\varepsilon)}\,dx\,\tilde{\pi}(dz,ds)                                           
+\int_0^t\int_Z\int_{\mathbb{R}^d}
J^{h^{(\varepsilon)}}|\bar u_{s-}^{\varepsilon}|^p\,dx\,\pi(dz,ds)       
\end{equation}
almost surely for all $t\in[0,T]$. Hence by Davis', Minkowski 
and H\"older inequalities, using standard estimates we obtain  
$$        
E\sup_{t\leq T}|\bar u_t^{\varepsilon}|^p_{L_p} 
\leq 2\,E|\psi^{(\varepsilon)}|_{L_p}^p
+NE\int_0^T|h_t^{(\varepsilon)}|^p_{L_p(\cL_p)}\,dt
+NT^{p-1}
|f^{0(\varepsilon)}|^p_{\bL_p}                                                              \nonumber\\
$$
\begin{equation}                                                               \label{Ito Lp estimate smooth}
+ NT^{(p-2)/2}
\left(|g^{(\varepsilon)}|^p_{\bL_p}
+E\int_0^T|h_t^{(\varepsilon)}|^p_{L_p(\cL_2)}\,dt
+\sum_{\alpha=1}^d|f^{\alpha(\varepsilon)}|_{\bL_p}^p
+\sum_{\alpha=1}^d|D_{\alpha}u^{(\varepsilon)}|^p_{\bL_p}
\right)                                                  
\end{equation}
with a constant $N=N(p,d)$, where 
$f^{\alpha(\varepsilon)}:=(f^{1\alpha(\varepsilon)},..., f^{M\alpha(\varepsilon)})$, 
and recall that $|v|_{L_p}$ means the $L_p$-norm of 
$|(\sum_{i=1}^M|v^i|^2)^{1/2}|$ 
for $\bR^M$-valued functions 
$v=(v^1,...,v^M)$ on $\bR^d$. Hence  
$$
E\sup_{t\leq T}|\bar u_t^{\varepsilon}-\bar u_t^{\varepsilon'}|^p_{L_p}
\rightarrow 0
\quad
\text{as $\varepsilon,\,\varepsilon'\to0$}.  
$$ 
Consequently, there is an $L_p$-valued adapted cadlag process 
$\bar u=(\bar u_t)_{t\in[0,T]}$ 
such that 
$$
\lim_{\varepsilon\to0}E\sup_{t\leq T}|\bar u_t^{\varepsilon}-\bar u|^p_{L_p}=0. 
$$
Thus for each $\varphi\in C_0^{\infty}(\bR^d)$ we can take $\varepsilon\to0$ in 
\begin{align*}
(\bar u_t^{i\varepsilon},\varphi)
& =(\psi^{i(\varepsilon)},\varphi)+\int_0^t(f_s^{i(\varepsilon)},\varphi)\,ds
+\int_0^t(g_s^{ir(\varepsilon)},\varphi)\,dw_s^r
+\int_0^t\int_Z(h_s^{i(\varepsilon)},\varphi)\,\tilde{\pi}(dz,ds)\\
&
= (\psi^{i(\varepsilon)},\varphi)+\int_0^t(f_s^{i0(\varepsilon)},\varphi)\,ds
-\int_0^t(f_s^{ik(\varepsilon)},D_k\varphi)\,ds
+\int_0^t(g_s^{ir(\varepsilon)},\varphi)\,dw_s^r\\
&
\quad +\int_0^t\int_Z(h_s^{i(\varepsilon)},\varphi)\,\tilde{\pi}(dz,ds)
\end{align*}
and it is easy to see that we get 
$$
(\bar u^i_t,\varphi)=(\psi^i,\varphi)+\int_0^t(f_s^{i\alpha},D^{\ast}_{\alpha}\varphi)\,ds
+\int_0^t(g_s^{ir},\varphi)\,dw_s^r+\int_0^t\int_Z(h^i_s,\varphi)\,\tilde{\pi}(dz,ds) 
$$
almost surely for all $t\in[0,T]$. Hence $\bar u= u$ for $P\otimes dt$-almost every 
$(\omega,t)\in\Omega\times[0,T]$.  
Finally letting $\varepsilon\to0$ in \eqref{Lp Ito formula smooth} 
we obtain \eqref{Ito2}. 
\end{proof}

\end{document}